\pdfoutput=1
\documentclass{amsart}
\usepackage{
 amsmath, 
 amsxtra, 
 amsthm, 
 amssymb, 
 etex, 
 mathrsfs, 
 mathtools, 
 tikz-cd, 
 xr}
 \usepackage{hyperref}

 \usepackage{dsfont}

 \usepackage[utf8]{inputenc}
 \usepackage[T1]{fontenc}
 
\usepackage{helvet}
\usepackage{color}

\usepackage{pifont}
\usepackage[bbgreekl]{mathbbol}
\usepackage{amsfonts}
\usepackage{mathrsfs}

\usepackage[all]{xy}

\usepackage{comment}

\newtheorem{theorem}{Theorem}[section]
\newtheorem{lemma}[theorem]{Lemma}

\newtheorem{proposition}[theorem]{Proposition}
\newtheorem{corollary}[theorem]{Corollary}
\newtheorem{defn}[theorem]{Definition}
\newtheorem{example}[theorem]{Example}

\newtheorem{notation}[theorem]{Notation}

\newtheorem{lthm}{Theorem} 

\theoremstyle{remark}
\newtheorem{remark}[theorem]{Remark}

\newcommand{\Tr}{\operatorname{Tr}}
\newcommand{\Gal}{\operatorname{Gal}}
\newcommand{\Fil}{\operatorname{Fil}}
\newcommand{\Sym}{\operatorname{Sym}}

\newcommand{\DD}{\mathbb{D}}

\newcommand{\QQ}{\mathbb{Q}}
\newcommand{\Qp}{\mathbb{Q}_p}
\newcommand{\Zp}{\mathbb{Z}_p}
\newcommand{\ZZ}{\mathbb{Z}}

\newcommand{\FFF}{\mathcal{F}}

\newcommand{\vp}{\varphi}

\newcommand{\cO}{\mathcal{O}}
\newcommand{\cN}{\mathcal{N}}

\newcommand{\col}{\mathrm{Col}}

\newcommand{\cH}{\mathcal{H}}

\newcommand{\LL}{\Lambda}
\newcommand{\TT}{\mathbb{T}}

\newcommand{\lra}{\longrightarrow}
\newcommand{\ra}{\rightarrow}
\newcommand{\res}{\textup{res}}
\newcommand{\cs}{\clubsuit}
\newcommand{\cor}{\mathrm{cor}}

\newcommand{\BF}{\textup{BF}}

\newcommand{\Dcris}{\DD_{\rm cris}}
\newcommand{\cE}{\mathcal{E}}

\begin{document}

\title{Rank--two Euler systems for symmetric squares}

\begin{abstract}
Let $p\ge 7$ be a prime number and $f$ a normalized eigen-newform with good reduction at $p$ such that its $p$-th Fourier coefficient vanishes. We construct a rank-two Euler system attached to the $p$-adic realization of the symmetric square motive of $f$. Furthermore, we show that the non-triviality is guaranteed by the non-vanishing of the leading term of the relevant $L$-value and the non-vanishing of a certain $p$-adic period modulo $p$.
\end{abstract}
\author[K. B\"uy\"ukboduk]{K\^az\i m B\"uy\"ukboduk}
\address[B\"uy\"ukboduk]{UCD School of Mathematics and Statistics\\ University College Dublin\\ Ireland}
\email{kazim.buyukboduk@ucd.ie}

\author[A. Lei]{Antonio Lei}
\address[Lei]{D\'epartement de Math\'ematiques et de Statistique\\
Universit\'e Laval, Pavillion Alexandre-Vachon\\
1045 Avenue de la M\'edecine\\
Qu\'ebec, QC\\
Canada G1V 0A6}
\email{antonio.lei@mat.ulaval.ca}

\maketitle
\tableofcontents

\section{Introduction}
The main goal of this article is to construct  a rank-two Euler system in the setting of \cite{BLV}. We fix forever a prime number $p \geq 7$ as well as embeddings $\iota_\infty:\overline{\QQ}\hookrightarrow \mathbb{C}$ and $\iota_p:\overline{\QQ}\hookrightarrow \mathbb{C}_p$. Let  $f$ be a normalised,  cuspidal new eigenform of weight $k + 2$, level $N$ and nebentype $\epsilon_f$. Throughout this article, we assume that {$p> k+1$}, $p \nmid N$ and $a_p(f) = 0$. We  write $\pm\alpha$ for the roots of the Hecke polynomial $X^2+\epsilon_f(p)p^{k+1}$ of $f$ at $p$.

Let $L/\QQ$ be a number field containing the Hecke field of $f$ as well as the $N$-th roots of unity. Let $\mathfrak{p}$ be a prime in $L$ above $p$. We denote by $E$ the completion of $L$ at $\mathfrak{p}$. Let $\cO$ denote the ring of integers of $E$ and write $W_f$ for Deligne's $E$-linear $G_\QQ$-representation   attached to $f$. Since $p>k+1$ by assumption, there exists a canonical Galois-stable $\cO$-lattice $R_f$ inside $W_f$ (c.f. \cite[\S14.22]{kato04} as well as \cite{FM}). Let us fix an even $E$-valued Dirichlet character $\chi$ of conductor $N_\chi$,  which we assume to be coprime to $p$. We will be interested in the $G_\QQ$-representations
\[
T:=\Sym^2R_f^*(1-j+\chi),\quad V:=\Sym^2 W_f^*(1-j+\chi)=T\otimes_{\Zp}\Qp,
\]
where $j\in [1,k+1]$ is an even integer.

 Let $\Gamma=\Gal(\QQ(\mu_{p^\infty})/\QQ)$. We write $\Gamma = \Gamma_{\text{tors}} \times \Gamma_{1}$, where $\Gamma_{\text{tors}}$ is a finite group of order $p-1$ and $\Gamma_{1} = \Gal(\QQ(\mu_{p^\infty})/\QQ(\mu_{p}))$. We fix a topological generator $\gamma$ of $\Gamma_{1}$, which in turn determines an isomorphism $\Gamma_{1} \cong \ZZ_p$.
 
 Let $\Lambda_\cO(\Gamma)$ denote the Iwasawa algebra $\cO[[\Gamma]]$ (and similarly, $\Lambda_\cO(\Gamma_1)$ the Iwasawa algebra $\cO[[\Gamma_1]]$). We write $\TT$ for the cyclotomic deformation
 \[
 \TT:=T\otimes_\cO \Lambda_\cO(\Gamma_1)^\iota,
 \]
where $\Lambda_{\cO}(\Gamma_1)^{\iota}$ denotes the free rank-one $\Lambda_{\cO}(\Gamma_1)$-module on which $G_{\QQ}$ acts via the inverse of the canonical character $G_{\QQ} \twoheadrightarrow \Gamma_1 \hookrightarrow \Lambda_{\cO}(\Gamma_1)^{\times}$. We similarly define $\Lambda_{\cO}(\Gamma)^{\iota}$. If $r\ge0$ is a real number, $\cH_{E,r}(\Gamma)$ denotes the algebra of  $E$-valued  tempered distributions of order $r$ on $\Gamma$. We define $\cH_{E,r}(\Gamma)^\iota$ in a manner similar to above.

Given an integer $r\ge1$, we write $\QQ(r)$ for the maximal $p$-extension of $\QQ$ inside $\QQ(\mu_r)$.  {Let $\mathcal{N}_{\chi}$ denote the set of positive square-free integers whose prime factors are prime to $6pNN_{\chi}$ and satisfy the conditions as specified in Definition~\ref{def:Kolyvaginprimesfortwists} below.}

As explained in \cite[\S8]{LLZ1}, conjectures of Perrin-Riou in \cite{pr-es} predict the existence of a rank-two Euler system for the representation $V$. A rank-two Euler system is a family of classes
\[
z_1^{(r)}\wedge z_2^{(r)}\in \bigwedge^2 H^1(\QQ(r),\TT),\quad r\in \mathcal{N}_{\chi},
\]
satisfying certain norm relations under the corestriction map as $r$ varies. In \cite{LZ1}, Loeffler and Zerbes constructed families of Beilinson-Flach elements attached to $T$. These classes live inside $H^1(\QQ(r),\TT\hat\otimes \cH_{E,k+1}(\Gamma))$ and as such, they are not integral. However, they do satisfy the Euler system norm relations. One expects that these classes are shadows of a (bounded) rank-two Euler system; in more precise terms, one expects that these unbounded classes arise from a rank-two Euler system via appropriate Perrin-Riou functionals. Evidence towards this was given in \cite{BLLV,BLV}.

 In this article, with the aid of ``signed'' Beilinson-Flach elements constructed in  \cite{BLV}, we give a rather explicit construction of a rank-two Euler system for $T$   under the following  hypotheses.
\begin{itemize}
\item[($\Psi_1$)] There exists $u \in (\ZZ/NN_\chi\ZZ)^\times$ such that $\epsilon_f\chi(u) \not\equiv \pm1\, (\textup{mod}\, \mathfrak{p})$ and $\chi(u)$ is a square modulo $\mathfrak{p}$. 
\item[($\Psi_2$)] $\epsilon_f\chi(p)\neq 1$ and $\phi(N)\phi(N_\chi)$ is coprime to $p$, where $\phi$ is Euler's totient function. 
\item[\textup{(\textbf{Im})}] $\textup{im}\left(G_\QQ\ra \textup{Aut}(R_f\otimes\QQ_p)\right)$ contains a conjugate of $\textup{SL}_2(\ZZ_p)$. 
\item[\textup{(\textbf{SD})}] The dual canonical Selmer group $H^1_{\mathcal{F}_{\textup{can}}^*}(\QQ,T^\vee(1))=0$ vanishes. 
\end{itemize}
Here, $T^\vee$ denotes the Pontryagin dual of $T$. Also, see Definition~\ref{def:canonicalselmerstructure} below, where we define the canonical Selmer structure and its dual.

\begin{lthm}[Corollary~\ref{cor:ESrank2}]\label{Thm_A}
Under the hypotheses $(\Psi_1)$, $(\Psi_2)$, $($\textbf{\textup{Im}}$)$ and $($\textbf{\textup{SD}}$)$, a non-zero rank-two Euler system for $T$ exists. In more precise terms, there exists a non-trivial collection 
$z^{(r)}\in \bigwedge^2 H^1(\QQ(r),\TT)$ as $r$ varies in $\mathcal{N}_{\chi}$ such that, for every $r\ell \in \mathcal{N}_{\chi}$ $($with $\ell\in \mathcal{N}_{\chi}$ a prime$)$ we have
$${\rm cor}_{\QQ(r\ell)/\QQ(r)}^{\otimes 2}(z^{(r\ell)})=P_\ell({\rm Fr}_\ell)\,z^{(r)}$$
where ${\rm Fr}_\ell$ is the geometrically normalised Frobenius at the prime $\ell$ and $P_\ell(X)=\det \left(1-{\rm Fr}_\ell X\vert V^*(1)\right)$ is the $\ell$-local Euler polynomial for $V^*(1):={\rm Hom}(V,\QQ_p(1))$.
\end{lthm}
\begin{remark}
The rank-two Euler system with the indicated  norm relation is deduced from the rank-two Euler system given by Corollary~\ref{cor:ESrank2} below 
  as follows: The Euler system of Corollary~\ref{cor:ESrank2} satisfies a norm relation  given by $Q_\ell(\ell^{j}{\rm Fr}_\ell)$, where $Q_\ell(X)$ is the Hecke polynomial for the Rankin-Selberg convolution $f\otimes f\otimes \chi^{-1}$ at the prime $\ell$. This is a degree $4$ poynomial and it factors as 
 $$Q_\ell(X)=(1-\ell^{k-1}\chi^{-1}\epsilon_f(\ell)X)\,P_\ell(\ell^{-j}X)\,.$$
 We can modify this Euler system following the proof of \cite[Theorem 5.3.3]{LZ2}, so that the new collection of cohomology classes satisfy the required norm relation.
\end{remark}
The hypotheses $(\Psi_1)$, $(\Psi_2)$ and $($\textbf{Im}$)$ of Theorem~\ref{Thm_A} were already present in \cite{BLV}. The hypothesis (\textbf{SD}) allows us to show that the corestriction map 
$$H^1(\QQ(r),\TT)\lra H^1(\QQ(r'),\TT)$$ 
is surjective for whenever $r'|r$ (see Proposition~\ref{prop:horizontalcontrol} below). In \S\ref{sec_trivialityofSelmer}, we give explicit criteria for (\textbf{SD}) to hold. One of them is in terms of Perrin-Riou's $p$-adic period and the leading term of the relevant $L$-function (Theorem~\ref{THM_nonzeromodperiodimpliesBFmodpnonzero}) and the other is in terms of the Rankin-Selberg $p$-adic $L$-function (Lemma~\ref{lem:NumericalCriterion}). 

Perrin-Riou's $p$-adic Beilinson conjecture further predicts a link between twists of the rank-two Euler system we construct and the second order derivative of  appropriate $L$-functions. This seems far out of reach with the current technology. That said,  the analogous  problem in the context of absolute Hodge cohomology (namely, Beilinson's conjectures themselves) seems more accessible and we will pursue this direction in future work.

\section{Beilinson-Flach elements and Equivariant $p$-adic $L$-functions}

 For $\lambda,\mu\in\{\pm\alpha\}$, $c > 1$ coprime to $6Np$, $m \ge 1$ coprime to $pc$, and $a\in (\ZZ / mp^\infty \ZZ)^\times$, let
   \[
    {}_c\BF^{\lambda,\mu}_{m, a} \in
    H^1(\QQ(\mu_m),W_{f}^* \otimes W_{f}^*\otimes\cH_{E,k+1}(\Gamma)^\iota)
   \]
   be the Beilinson--Flach element constructed in \cite[Theorem 5.4.2]{LZ1}. We shall take $a=1$ throughout and omit it from the notation. Let $\mathcal{R}_{\chi}$ denote the set of positive square-free integers prime to $6pcNN_{\chi}$. For $m \in \mathcal{R}_{\chi}$, we may consider $\chi$ as a continuous character of $\Gal(\QQ(\mu_{mN_{\chi}p^\infty})/\QQ) \cong(\ZZ/mN_{\chi}p^{\infty}\ZZ)^{\times}$.

For a square-free integer $m$ prime to $p$, we write $\QQ(m)$ for the maximal pro-$p$-extension of $\QQ$ contained in $\QQ(\mu_m)$. The Galois group of $\QQ(m)/\QQ$ is denoted by $\Delta_m$.

\begin{defn}
\label{define:twistedBF}
 For all $m \in \mathcal{R}_{\chi}$, we define the Beilinson-Flach class twisted by $\chi$
  $${}_{c}\BF_{m,\chi}^{\lambda,\mu}\in H^1(\QQ(m),W_{f}^* \otimes W_{f}^*(1 -j+ \chi)\otimes\cH_{E,k+1}(\Gamma)^\iota)$$ 
  by setting it as the image of ${}_{c}\BF_{m}^{\lambda,\mu}$ under the chain of maps
  \begin{align*}H^1(\QQ(\mu_{mN_\chi}),W_{f}^* \otimes W_{f}^*\otimes\cH_{E,k+1}(\Gamma)^\iota)&\cong H^1(\QQ(\mu_{mN_\chi}),W_{f}^* \otimes W_{f}^*(1 -j+ \chi)\otimes\cH_{E,k+1}(\Gamma)^\iota)\\
  &\stackrel{\textup{cor}}{\lra}H^1(\QQ(m),W_{f}^* \otimes W_{f}^*(1-j + \chi)\otimes\cH_{E,k+1}(\Gamma)^\iota)\\
&{\lra}H^1(\QQ(m),W_{f}^* \otimes W_{f}^*(1-j + \chi)\otimes\cH_{E,k+1}(\Gamma)^\iota)^+
  \end{align*}
  where the superscript $+$ stands for the $+1$-eigenspace under the action of the complex conjugation. 
  \end{defn}
We recall from \cite[Proposition~2.1.3 and Corollary~3.3.6]{BLV} that our choice of $\chi$ implies  that for $\lambda,\mu\in\{\alpha,-\alpha\}$, we have
\[
{}_{c}\BF_{m,\chi}^{\lambda,\lambda}, {}_{c}\BF_{m,\chi}^{\lambda,\mu}+{}_{c}\BF_{m,\chi}^{\mu,\lambda}\in H^1(\QQ(m),V\otimes \cH_{E,k+1}(\Gamma)^\iota)^{+},
\]
and 
\[
{}_{c}\BF_{m,\chi}^{\lambda,\mu}={}_{c}\BF_{m,\chi}^{\mu,\lambda}.
\]

From now on, we fix a choice of $c$ and discard the subscript $c$ from the notation (see \cite[Remark~2.2.8]{BLV}).

\section{Selmer group computations}Recall that $T=\Sym^2R_f^*(1-j+\chi)$ and $\TT=T\otimes_\cO \Lambda_\cO(\Gamma_1)^\iota$.
In this section, we consider the following set of Kolyvagin primes for the representation $T$:

\begin{defn}
\label{def:Kolyvaginprimesfortwists}
Let $\mathcal{P}_{\chi}$ denote the set of primes $\ell \nmid 6pcNN_{\chi}$ for which we have
\begin{itemize}
\item $\ell \equiv 1\mbox{ mod }p$,
\item $T/(\mathrm{Frob}_{\ell} - 1)T$ is a free $\cO$-module of rank one,
\item $\mathrm{Frob}_{\ell} - 1$ is bijective on $\bigwedge^2R_f^*(1-j+\chi)$.
\end{itemize}
We let $\mathcal{N}_{\chi}$ denote the set of square-free  integers whose prime factors all lie in $\mathcal{P}_{\chi}$.
\end{defn}
For each $r\in \cN_\chi$, we let $\LL_r$ denote the ring $\cO[[\Delta_r\times\Gamma_1]]$.
\begin{theorem}
\label{thm:mainSelmerstructure}For all $r\in\cN_\chi$, the $\LL_r$-module $H^1(\QQ(r),\TT)$ is free of rank $2$.
\end{theorem}
\begin{proof}
All the references in this proof are to \cite{BLV}, unless otherwise stated. The assertion is essentially Theorem~2.3.20, after minor modifications. We recall the twists $T_{\chi,j}:=T\otimes\omega^j$ of $R$ that we have utilized in op. cit. Notice that Proposition 2.3.2 remains valid when $T_{\chi,j}$ is replaced with $T$ with the same $\tau \in G_\QQ$, since we have $\omega(\tau)=1$ for $\tau$ verifying the conclusions of this proposition. Moreover, it is easy to see that the set of Kolyvagin primes $\mathcal{P}_\chi$ given as above is identical to the set of Kolyvagin primes for $T_{\chi,j}$ (given as in Definition 2.3.9 of op. cit.). Moreover, as in Remark~2.3.11, the existence of $\tau$ supplies us with a large set of choices for Kolyvagin primes for $T$. Using these facts,  the proofs in Section 2.3 of op. cit. go through in verbatim, with $T$ replacing $T_{\chi,j}$.
\end{proof}
\label{s:hori}

\begin{notation}
  For a module $M$, we denote its Pontryagin dual, $\textup{Hom}_{\textup{cts}}(M,\QQ_p/\ZZ_p)$, by $M^{\vee}$.
\end{notation}
\begin{defn}
\label{def:canonicalselmerstructure}
Let $K$ be any number field. Given an arbitrary free $\cO$-module $M$ of finite rank that is endowed with a continuous $G_K$-action unramified outside a finite set of places of $K$, we let $\mathcal{F}_{\textup{can}}$ denote the canonical Selmer structure on $M$ $($or $M\otimes_{\ZZ_p}\QQ_p$$)$, given as in \cite[Definition 3.2.1]{mr02}. We also let $\mathcal{F}_{\textup{can}}^*$ denote the dual Selmer structure on $M^\vee(1)$ $($or on $(M\otimes_{\ZZ_p}\QQ_p)^*(1)$$)$, defined as in Section 2.3 of op. cit. 

We shall also write $\mathcal{F}_{\textup{can}}$ for the Selmer structure on $\mathbb{M}:=M\otimes\Lambda_{\cO}(\Gamma)^\iota$, denoted by $\mathcal{F}_\LL$ in Section 5.3 of loc. cit. We shall write $\mathcal{F}_{\textup{can}}^*$  for the Selmer structure $\mathcal{F}_\LL^*$ of loc. cit. on the Galois representation $\mathbb{M}^\vee(1)$.
\end{defn}

Theorem~\ref{thm:mainSelmerstructure} tells us that $H^1(\QQ(r),\TT)$ is free over $\Lambda_r$. In order to choose compatible bases (as $r$ varies) for the collection of modules $H^1(\QQ(r),\TT)$, we would like to know whether there exists $R \in \mathcal{N}$ with the property that
$$\bigcap_{r\in \mathcal{N}} \textup{im}\left(H^1(\QQ(r),\TT) \stackrel{\textup{cor}}{\lra} H^1(\QQ,\TT)\right)=\textup{im}\left(H^1(\QQ(R),\TT) \stackrel{\textup{cor}}{\lra} H^1(\QQ,\TT)\right)\,.$$
We may give an affirmative answer to this question under mild hypothesis:
\begin{proposition}
\label{prop:horizontalcontrol}
Suppose that there exists a character $\rho$ of $\Gamma$ (that might be of infinite order) such that $H^1_{\mathcal{F}_{\textup{can}}^*}(\QQ,(T\otimes\rho)^*)=0$. Then for every $r,r^\prime \in \mathcal{N}_\chi$ with $r^\prime \mid r$, the corestriction map
$$H^1(\QQ(r),\TT)\lra H^1(\QQ(r^\prime),\TT)$$
is surjective.
\end{proposition}
\begin{proof}
 Fix $r \in \mathcal{N}_\chi$ and let $L\subset \QQ(r)$ denote any subfield. Set $\Delta_F:=\Gal(F/\QQ)$ and let $\mathcal{A}_{L} \subset \cO[\Delta_L]$ denote the augmentation ideal. We first verify that $H^1_{\mathcal{F}_{\textup{can}}^*}(L,\TT^\vee(1))=0$. It follows from \cite[Lemma 3.5.3]{mr02}\footnote{This lemma applies since the residual representation $\overline{T}$ is absolutely irreducible and we have $\overline{T}^{G_\QQ}=(\overline{T}^\vee(1))^{G_\QQ}=0$. Note that  these two conditions are all that are required in the proof of \cite[Lemma 3.5.3]{mr02}.} that for every character $\rho$ of $\Gamma$,
$$H^1_{\mathcal{F}_{\textup{can}}^*}(L,\TT^\vee(1))^\vee\big{/}(\mathcal{A}_{L},\rho^{-1}(\gamma)\gamma-1)\cong H^1_{\mathcal{F}_{\textup{can}}^*}(\QQ,(T\otimes\rho)^\vee(1))^\vee\,.$$
By assumption, we may choose  $\rho$ so that $H^1_{\mathcal{F}_{\textup{can}}^*}(\QQ,(T\otimes\rho)^\vee(1))^\vee=0$ and the required vanishing of $H^1_{\mathcal{F}_{\textup{can}}^*}(L,\TT^\vee(1))^\vee$ follows from Nakayama's lemma.

Let now $\QQ(r^\prime)=F_0\subset\cdots\subset F_s=\QQ(r)$ denote a chain of field extensions such that $F_i/F_{i-1}$ is cyclic for $i=1,\cdots, s$. Let us set $\Delta^{(i)}:=\Gal(F_i/F_{i-1})$ and let us write $\mathcal{A}^{(i)}\subset \cO[\Delta^{(i)}]$ for the augmentation ideal. We have an isomorphism
$${\rm coker}\left(H^1(F_i,\TT)\lra H^1(F_{i-1},\TT)\right)\stackrel{\sim}{\lra} H^2(F_i,\TT)[\mathcal{A}^{(i)}]\,.$$
By \cite[(8.9.6.2)]{nekovar06}, we also have a Matlis duality isomorphism $H^1_{\mathcal{F}_{\textup{can}}^*}(F_i,\TT^\vee(1))^\vee\cong H^2(F_i,\TT)$ and the proof follows from the vanishing we have established in the first paragraph.\end{proof}
We henceforth assume that 
\begin{equation}
\label{eqn:vanishingofstrictselmer}
H^1_{\mathcal{F}_{\textup{can}}^*}(\QQ,T^*)=0.
\end{equation}
In other words, the hypothesis (\textbf{SD}) in the introduction holds. Furthermore, Proposition~\ref{prop:horizontalcontrol} applies.
\begin{defn}
A collection $\left\{c_1^{(r)},c_2^{(r)}\right\}_{r\in \mathcal{N}_\chi}$ (where $\{c_1^{(r)},c_2^{(r)}\}$ is a $\LL_r$ basis for $H^1(\QQ(r),\TT)$) is said to be \emph{horizontally compatible} if 
$$\textup{cor}_{\QQ(r)/\QQ(r^\prime)}^{\otimes2}\left(c_1^{(r)}\wedge c_2^{(r)} \right)=c_1^{(r^\prime)}\wedge c_2^{(r^\prime)}$$ for  every $r^\prime\mid r$.
\end{defn}
{The surjectivity of the corestriction maps given by Proposition~\ref{prop:horizontalcontrol} allows us to choose a basis in $H^1(\QQ(r),\TT)$ for each $r$, forming a  horizontally compatible collection}. We once and for all fix one and denote it by  $\mathscr{C}:=\left\{c_1^{(r)},c_2^{(r)}\right\}_{r\in \mathcal{N}_\chi}$.

\section{Signed Iwasawa Theory}
For a finite unramified extension $F$ of $\Qp$ and  $\cs \in \{+,-, \bullet\}$, recall from \cite[\S3.2-\S3.4]{BLV} that we have defined the signed Coleman map
\[
\col^\cs_{F}:H^1 (F,\TT)\rightarrow\cO_{F}\otimes_{\Zp}\Lambda_\cO(\Gamma_1).
\]
These maps are norm compatible, that is, if $F'/F$ is a finite unramified extension, then
\[
\col^\cs_{F}\circ\cor_{F'/F}=\Tr_{F'/F}\circ \col^\cs_{F'},
\]
as a consequence of the corresponding property of Perrin-Riou's big logarithm map (c.f. the proof of Proposition~ 4.3 and Proposition~4.12 in \cite{LZ-IJNT}). Via a choice of basis of the Yager module or a normal basis of the residue field, we may identify $\cO_{F}$ with $\Zp[\Gal(F/\Qp)]$ (c.f. \cite[Proposition~3.2 and Remark~3.3]{LZ-IJNT}). Then, we may consider $\col^\cs_{F}$ as a map landing in $\cO[[\Gal(F/\Qp)\times\Gamma_1]]$.

Fix $r \in \mathcal{N}_\chi$ and let $v_1,\ldots, v_t$ be the places of $\QQ(r)$ lying above $p$. For each $i$, may identify $\QQ(r)_{v_i}$ with a finite unramified extension of $\Qp$, say $F$, which is independent of $i$. Suppose that $v_i=\sigma_i(v_1)$, where $\sigma_i\in\Delta_r$. Let $D$ be the decomposition group of $v_1$ in $\Delta_r$ (which is isomorphic to $\Gal(F/\Qp)$).  Then, $\Delta_r=\sqcup \sigma_i D$, which allows us to identify $\Lambda_r$ with $\sqcup_i\cO[[ D\times \Gamma_1]]$. Let $H^1(\QQ(r)_p,\TT)=\bigoplus_{v|p}H^1(\QQ(r)_v,\TT)$. We may define the semi-local signed Coleman map
\[
\col^\cs_r=\bigoplus_{v|p} \col^\cs_{\QQ(r)_v}: H^1(\QQ(r)_p,\TT)\rightarrow \cO[[\Delta_r\times \Gamma_1]].
\]
Furthermore, we have the following signed objects attached to this map:

\begin{itemize}
\item[(a)] The Selmer structures $\mathcal{F}_\cs$ (resp., the dual Selmer structure $\mathcal{F}_\cs^*$) on  $\TT$ (resp., on  $\TT^*$)  given by the local conditions at $p$ determined by $H^1_\cs(\QQ(r)_p,\TT):=\ker\left(\col_r^\cs\right)$ (resp., by $\ker\left(\col_r^\cs\right)^\perp$) and the conditions outside $p$ agree with those given by the canonical Selmer structure.
\item[(b)] The signed Beilinson-Flach classes $\{C\times\BF_{r,\chi}^{\cs}\}$ with $C\times\BF_{r,\chi}^{\cs} \in H^1_{\mathcal{F}_\clubsuit} (\QQ(r),\TT)$, where $C$ is a non-zero constant independent of $\cs$ and $r$ (c.f. \cite[Proposition~3.4.2]{BLV}).
\end{itemize}

\begin{proposition}
\label{prop:thenonvanishingsignedclass}
There exists a choice of $\cs\in  \{+,-,\bullet\}$ so that $\BF_{1,\chi}^\cs \neq 0$.
\end{proposition}
\begin{proof}
This is given by the proof of Proposition~3.5.4 in \cite{BLV}.
\end{proof}
\begin{theorem}
\label{thm:structureofsignedSelmergroups}
Let $\cs$ be as in Proposition~\ref{prop:thenonvanishingsignedclass}. The Selmer group $H^1_{\mathcal{F}_\cs^*} (\QQ,\TT^\vee(1))$ is $\LL$-cotorsion and $H^1_{\mathcal{F}_\cs} (\QQ,\TT)$ is torsion free of rank one.
\end{theorem}
\begin{proof}
The first part follows from \cite[Theorem 5.3.6]{mr02} applied with the Euler system $\left\{\BF_{r,\chi}^{\cs}\right\}$ and Proposition~\ref{prop:thenonvanishingsignedclass}, whereas the second part is an immediate consequence of the first part together with Poitou-Tate global duality. We note that $H^1_{\mathcal{F}_\cs} (\QQ,\TT)$ is torsion free since the Iwasawa cohomology $H^1(\QQ,\TT)$ is thanks to Theorem~\ref{thm:mainSelmerstructure}.
\end{proof}

For the rest of the article, we fix a choice of $\cs\in\{+,-,\bullet\}$ that satisfies Proposition~\ref{prop:thenonvanishingsignedclass}.

\begin{defn}For all $r\in\cN_\chi$,
 we let $\mathscr{F}_r$ denote the compositum of the arrows
$$H^1(\QQ(r),\TT)\stackrel{\res_p}{\lra} H^1(\QQ(r)_p,\TT)\stackrel{\col_r^\cs}{\lra} \LL_r\,.$$
We also define the $\LL_r$-morphism
$$\mathscr{G}_r: \bigwedge^2 H^1(\QQ(r),\TT)\lra H^1(\QQ(r),\TT)$$ 
by setting 
$$\mathscr{G}_r(z_1\wedge z_2):=\mathscr{F}_r(z_1)z_2-\mathscr{F}_r(z_2)z_1\,.$$
\end{defn}
Recall the horizontally compatible bases $\mathscr{C}:=\left\{c_1^{(r)},c_2^{(r)}\right\}_{r\in \mathcal{N}_\chi}$ from \S\ref{s:hori}.
\begin{proposition}
\label{prop:Frisnotidenticallyzerooncr}
For all $r\in\cN_\chi$, we have either $\mathscr{F}_r\left(c_1^{(r)}\right)\neq 0$ or $\mathscr{F}_r\left(c_2^{(r)}\right)\neq 0$.
\end{proposition}
\begin{proof}
If we had 
$$\mathscr{F}_r\left(c_1^{(r)}\right)=\mathscr{F}_r\left(c_2^{(r)}\right)= 0$$ 
it would follow that the map $\mathscr{F}_r$ is identically zero. This in turn implies that 
$$H^1_{\mathcal{F}_\cs}(\QQ(r),\TT)=H^1(\QQ(r),\TT)$$
is a free $\LL_r$-module of rank $2$. Using the trace-compatibility of the Coleman maps $\col_r^{\cs}$ as $r$ varies, we observe that the corestriction map induces an injection 
$$H^1_{\mathcal{F}_\cs}(\QQ(r),\TT)/\mathcal{A}_r\cdot H^1_{\mathcal{F}_\cs}(\QQ(r),\TT) \hookrightarrow H^1_{\mathcal{F}_\cs}(\QQ,\TT),$$
where $\mathcal{A}_{r} \subset \cO[\Delta_r]$ is the augmentation ideal. This shows that the module $H^1_{\mathcal{F}_\cs}(\QQ,\TT)$ contains a free $\LL_\cO(\Gamma)$-module of rank $2$, contradicting Theorem~\ref{thm:structureofsignedSelmergroups}. \end{proof}

\begin{corollary}
\label{cor:Grinjective}
The $\LL_r$-homomorphism $\mathscr{G}_r$ is injective.
\end{corollary}
\begin{proof}
Since the $\LL_r$-module $\bigwedge^2 H^1(\QQ(r),\TT)$ is free of rank $1$ and as the target module $H^1(\QQ(r),\TT)$ is torsion-free (by Theorem~\ref{thm:mainSelmerstructure}), we only need to verify the map $\mathscr{G}_r$ is not identically zero. 

Recall the basis $\left\{c_1^{(r)},c_2^{(r)}\right\}$ of $H^1(\QQ(r),\TT)$ and suppose that  $\mathscr{G}_r$ is the zero map. This in particular means that 
$$\mathscr{G}_r(c_1^{(r)}\wedge c_2^{(r)}):=\mathscr{F}_r(c_1^{(r)})\cdot c_2^{(r)}-\mathscr{F}_r(c_2^{(r)})\cdot c_1^{(r)}=0$$
and therefore that 
$$\mathscr{F}_r(c_1^{(r)})=\mathscr{F}_r(c_2^{(r)})=0\,.$$
This contradicts Proposition~\ref{prop:Frisnotidenticallyzerooncr}.
\end{proof}

\begin{lemma}
$\textup{im}\left(\mathscr{G}_r\right) \subset H^1_{\mathcal{F}_\cs} (\QQ(r),\TT)$.
\end{lemma}
\begin{proof}
This follows from the observation that 
$$\mathscr{F}_r\circ\mathscr{G}_r(z_1\wedge z_2)=\mathscr{F}_r(z_1)\mathscr{F}_r(z_2)-\mathscr{F}_r(z_2)\mathscr{F}_r(z_1)=0\,.$$
\end{proof}
We shall need the following elementary linear algebra result in the proof of Theorem~\ref{thm:BFelementsareintheimage}.
\begin{lemma}
\label{lem:linearalgebra}
Suppose $R$ is any commutative ring and $M$ is a free $R$-module of rank $2$, with basis $\{u,v\}$. Let $a,b\in R$ and $K=R\cdot (au-bv)$ be the submodule of $M$ generated by $au-bv$. Let $N \subset M$ be a submodule that contains $K$ and for which we have $N\cap Rv=\{0\}$. Then $aN \subseteq K$.
\end{lemma}
\begin{proof}
Let $\iota$ denote the natural injection $N \stackrel{\iota}{\hookrightarrow} M/Rv$ (given by $ n \mapsto n +Rv$), induced by the containment $N\subseteq M$; we shall denote its restriction to $K\subseteq N$ also by $\iota$. Let $x=su+tv \in N$ be any element. We then have 
$$\iota(ax)=\iota(sau)+\iota(tav)=s\iota(au)=s\iota(au-bv) \in \iota (K).$$
Since $\iota$ is injective, it follows that $ax \in K$ as desired.
\end{proof}
\begin{theorem}
\label{thm:BFelementsareintheimage}
For $i=1,2$ and $r\in \cN_\chi$, we have
$$\mathscr{F}_r\left(c_i^{(r)}\right)\cdot \BF_{r,\chi}^\cs \in \textup{im}\left(\mathscr{G}_r\right)\,.$$
\end{theorem}
\begin{proof}
We first consider the case when $\mathscr{F}_r\left(c_1^{(r)}\right)\neq 0\neq \mathscr{F}_r\left(c_2^{(r)}\right)$\,. In this situation, notice that we have 
$$\LL_r\cdot c_i^{(r)}\,\cap\, H^1_{\mathcal{F}_\cs}(\QQ(r),\TT)=\{0\}\,$$
(where $i=1,2$) and therefore a chain of containments of $\LL_r$-modules
\begin{align}\label{eqn:containments}
\textup{im}(\mathscr{G}_r)=\LL_r\cdot \left(\mathscr{F}_r(c^{(r)}_2)c_1^{(r)}-\mathscr{F}_r(c^{(r)}_1)c_2^{(r)}\right)&\subseteq H^1_{\mathcal{F}_\cs}(\QQ(r),\TT) \\
\notag&\hookrightarrow H^1(\QQ(r),\TT)/\LL_r\cdot c_i^{(r)}\,.
\end{align}
Let $j$ denote the element of $\{1,2\}$ other than $i$ (so that $\{i,j\}=\{1,2\}$). For each $c\in H^1_{\mathcal{F}_\cs}(\QQ(r),\TT)$, we infer from (\ref{eqn:containments}) and Lemma~\ref{lem:linearalgebra} that  
$$\mathscr{F}_r(c^{(r)}_j)\cdot c \in \textup{im}(\mathscr{G}_r)$$
and the proof of our theorem is complete in this case.

Suppose next that we have $\mathscr{F}_r\left(c_i^{(r)}\right)=0$ for some $i \in \{1,2\}$; say without loss of generality $i=1$. Observe that this is equivalent to saying that $c_1^{(r)} \in H^1_{\mathcal{F}_\cs}(\QQ(r),\TT)$. It follows from Proposition~\ref{prop:Frisnotidenticallyzerooncr} that $\mathscr{F}_r\left(c_2^{(r)}\right)\neq 0$, or in other words, that 
\begin{equation}\label{eqn:signedvsc2transversal}
\LL_r\cdot c_2^{(r)} \,\cap\, H^1_{\mathcal{F}_\cs}(\QQ(r),\TT) =\{0\}\,.
\end{equation}
Notice further that the assertion that  $\mathscr{F}_r(c_1^{(r)})\cdot \BF_{r,\chi}^\cs \in \textup{im}(\mathscr{G}_r)$ trivially holds and it remains to verify that 
\begin{equation}
\label{eqn:thedesiredcontainment}
\mathscr{F}_r(c_2^{(r)})\cdot \BF_{r,\chi}^\cs \in \textup{im}(\mathscr{G}_r)=\LL_r \left(\mathscr{F}_r(c_2^{(r)})\cdot c_1^{(r)}\right)\,.
\end{equation}

The natural containment maps and (\ref{eqn:signedvsc2transversal}) induce the commutative diagram
$$\xymatrix{ \LL_r\cdot c_1^{(r)}\ar@{}[r]|-*[@]{\subseteq}\ar@{->>}[rd]& H^1_{\mathcal{F}_\cs}(\QQ(r),\TT)\ar@{^{(}->}[d]\\
 &H^1(\QQ(r),\TT)\big{/}\LL_r\cdot c_2^{(r)}}$$
 and it follows that $H^1_{\mathcal{F}_\cs}(\QQ(r),\TT)=\LL_r\cdot c_1^{(r)}$ and in particular also that $\BF_{r,\chi}^\cs \in \LL_r\cdot c_1^{(r)}\,.$ The desired containment (\ref{eqn:thedesiredcontainment}) follows at once, concluding the proof of our theorem.
\end{proof}

\begin{defn}
For each $r \in \mathcal{N}_\chi$, we set 
$$z_{1,i}^{(r)}\wedge z_{2,i}^{(r)}:=\mathscr{G}_r^{-1}\left(\mathscr{F}_r\left(c_i^{(r)}\right)\cdot \BF_{r,\chi}^\cs\right)\in\bigwedge^2 H^1(\QQ(r),\TT) \,.$$
Note that the element $z_{1,i}^{(r)}\wedge z_{2,i}^{(r)}$ is well-defined thanks to Corollary~\ref{cor:Grinjective}.
\end{defn}

\begin{corollary}
\label{cor:ESrank2}
For at least one choice of $i=1,2$, the collection 
$$\left\{z_{1,i}^{(r)}\wedge z_{2,i}^{(r)}\right\}_{r\in \cN_\chi}$$
is a non-trivial Euler system of rank 2 for the Galois representation $\TT$.
\end{corollary}
\begin{proof}
This follows from Proposition~\ref{prop:Frisnotidenticallyzerooncr}, Corollary~\ref{cor:Grinjective}, the fact that the collection $\left\{\BF_{r,\chi}^\cs\right\}_{r\in \mathcal{N}_\chi}$ is a non-trivial Euler system of rank one, the trace-compatibility of the maps $\mathscr{F}_r$ (as $r$ varies) and the defining property of the horizontally compatible bases $\mathscr{C}$.
\end{proof}
\section{Triviality of Selmer groups}
\label{sec_trivialityofSelmer}
Our goal in this section is to provide sufficient conditions to ensure the validity of \eqref{eqn:vanishingofstrictselmer}. Throughout this section, we will assume the truth of the following condition (in addition to ($\Psi_1$), ($\Psi_2$) and ({\bf Im}) above):

($\Psi_3$) The prime $\frak{p}$ has degree $1$ and ${\rm im}(\chi)\subset \ZZ_p^\times$.
\begin{defn}
\label{def_nonpstabilizedandmodpBFclass}
Let $\BF_1 \in H^1(\QQ,T)$ denote the (non-$p$-stabilized) Beilinson-Flach class and let $\overline{\BF}_1 \in H^1(\QQ,\overline{T})$ its image under reduction modulo the maximal ideal $\frak{m}$ of $\cO$.
\end{defn}

Note that $\BF_1 $  indeed lies inside $H^1(\QQ,T)$ thanks to our choice of $j$ (c.f \cite[Definition~3.2.2]{LZ1}).
\begin{proposition}
\label{prop_BFnontrivialmodpimpliestrivialselmer}
If $\overline{\BF}_1$ is non-zero, then \eqref{eqn:vanishingofstrictselmer} holds true.
\end{proposition}
\begin{proof}
Let $\kappa^{\BF} \in \overline{\bf{KS}}(T,\FFF_{\rm can},\mathcal{P}_\chi)$ denote the  (non-$p$-stabilized) Kolyvagin system that the Beilinson-Flach classes descend to (via Theorem 3.2.4 of \cite{mr02}). Here, $\FFF_{\rm can}$ is the canonical Selmer structure given as in Definition 3.2.1 of op. cit. The condition that $\overline{\BF}_1\neq 0$ translates into $\partial^{(0)}(\kappa^\BF)=0$ in the notation of \cite[Definition 5.2.1]{mr02}. The result follows thanks to Theorem 5.2.2 of op. cit. (as modified slightly in \cite[Appendix 1]{LZ2} to cover our particular case of interest). Note that all hypotheses required to apply the results in \cite{mr02} that we refer to hold thanks to our running assumptions; c.f. the paragraph following the statement of \cite[Theorem 5.3.4]{LZ2} where this is confirmed in a closely related setting.
\end{proof}
{
From now on, we assume that the Fontaine-Laiffaille condition holds for the representation $R_f^*\otimes R_f^*(1-j+\chi)$. That is:
\begin{itemize}
\item[\textup{(\textbf{FL})}] $p\ge 2k+3$. 
\end{itemize}
}
\begin{lemma}\label{lem:isoBK}
For all integers $j\in [1,k+1]$, the Bloch-Kato exponential map gives an isomorphism 
\[
\frac{\Dcris(R_f^*\otimes R_f^*(1-j+\chi))}{\Fil^0}\stackrel{\sim}{\longrightarrow}H^1_{\rm f}\left(\Qp,R_f^*\otimes R_f^*(1-j+\chi)\right).
\]
\end{lemma}
\begin{proof}
Recall that the Hodge-Tate weights of $R_f^*\otimes R_f^*(1-j+\chi)$ are $1-j,k-j+2,k-j+2$ and $2k-j+3$. In particular, $0,1\in [1-j,2k-j+3]$. Furthermore, (\textbf{FL}) says that
\[
p-1\ge 2k+2=(2k-j+3)-(1-j).
\]
Consequently, \cite[Proposition~II.2]{berger02} tells us that the ($p$-adic) Tamagawa number for the lattice $R_f^*\otimes R_f^*(1-j+\chi)$ is a $p$-adic unit. 

The decomposition of the local representation $R_f^*\otimes R_f^*(1-j+\chi)|_{G_{\Qp}}$ as given in \cite[\S3.1]{BLV} and our assumption ($\Psi_1$) imply that 
\[
H^0(\Qp,R_f^*\otimes R_f^*(1-j+\chi))=0.
\]
Therefore, the fact that the Tamagawa number is a unit implies that the Bloch-Kato exponential map is an isomorphism.
\end{proof}

\begin{defn} {Let $\omega_f\in \Dcris(W_f)$  be the canonical vector given in the beginning of \S6.1 in  \cite{KLZ1} (\emph{not} $\omega_f'$ in Lemma~6.1.1 of op. cit.)  and let  $\eta_f^{\pm\alpha}\in \Dcris(W_f)$ be the unique $\pm\alpha$-eigenvector lifting the vector $\eta_f$ as given in Remark~6.1.4 in \textit{op. cit}.} \end{defn}

As explained in \cite[\S14.22]{kato04}, $\omega_f$ is an $\cO$-basis of the filtration $\Fil^1\Dcris (R_f)$ under our assumption that $p>k+1$. Consequently, $\omega_f,p^{-k-1}\vp(\omega_f)$  form an $\cO$-basis of $\Dcris(R_f)$ (for example, we may take the twists of the basis obtained in \cite[Lemma~3.1]{LLZ-CJM}).

For $\lambda\in\{\pm\alpha\}$, the vector  $\eta_f^{\lambda}$ is a $\vp$-eigenvector and it  pairs with $\omega_{f^*}$ to $1$, where $f^*$ is the conjugate of $f$. Thus, the image of  $\eta_f^{\lambda}$ in $\Dcris(W_f)/\Fil^1$ gives an $\cO$-basis of $\Dcris(R_f)/\Fil^1$. Note that  $\eta_f^{\lambda}$ is, up to a $\cO$-unit, given by $p^{-k-1}(\vp(\omega_f)+\lambda\omega_f)$. In particular, it does not live inside $\Dcris(R_f)$.

By an abuse of notation, we shall denote the natural image of $\eta_f^{\lambda}\otimes\omega_f$ in $\Fil^0\Dcris(W_f\otimes W_f(j-\chi))$ again by $\eta_f^{\lambda}\otimes\omega_f$.

\begin{defn}
Let $W$ be a finite dimensional $E$-vector space and $W^\circ\subset W$ an $\cO$-lattice. Given $x\in W$, we shall set 
$${||x||}_{W^\circ}:=p^{-\min\{v_p(r): r\in E \hbox{ with } rx \in W^\circ\}}\,.$$
Here, $v_p$ is the $p$-adic valuation on $\overline{\QQ}_p$ normalized so that $v_p(p)=1$.
\end{defn}
\begin{defn}\item[(i)] For $\lambda\in\{\pm\alpha\}$, we set 
$$\mathcal{N}_{\lambda}:={\rm span}_{E}(\eta_f^\lambda)\otimes\Dcris(W_f(j-\chi)) \subset \Dcris(W_f\otimes W_f(j-\chi))\,.$$
\item[(ii)] We set $\omega_{\rm tg}=\omega_{f^*}\otimes \omega_{f^*\otimes\chi}$ and let 
$$\langle\omega_{\rm tg}\rangle^{\perp}_\lambda:=\{v\in \cN_\lambda: \langle \omega_{\rm tg},v\rangle =0\}$$
denote the orthogonal complement of the line spanned by the ``tangent vector'' $\omega_{\rm tg}$. We define the lattice
$$\langle\omega_{\rm tg}\rangle^{\perp,\circ}_\lambda:=\langle\omega_{\rm tg}\rangle^{\perp}_\lambda\,\cap\, \Fil^0\Dcris(R_f\otimes R_f(j-\chi))\,.$$

\item[(iii)]We define the lattice $\cN^\circ_\lambda\subset \cN_\lambda$ by setting
$$\mathcal{N}_{\lambda}^\circ:=\left\{v\in \cN_\lambda: ||\langle \omega_{\rm tg},v\rangle||_p\leq 1 \hbox{ \,\,and \,\,} ||v-\langle \omega_{\rm tg},v\rangle v ||_{\langle\omega_{\rm tg}\rangle^{\perp,\circ}_\lambda}\leq 1\right\}.$$ 
We call $\cN^\circ_\lambda$ the optimization of the lattice $\mathcal{N}_{\lambda}\cap \Dcris(R_f\otimes R_f(j-\chi))$ $($which it contains$)$ relative to $\omega_{\rm tg}$.
\end{defn}
\begin{remark}\label{rem_basisfortheoptimisedlattice}
It is not hard to give an explicit description of the lattices $\langle\omega_{\rm tg}\rangle^{\perp,\circ}_\lambda$ and $\mathcal{N}_{\lambda}^\circ$ in terms of the distinguished vectors of $\Dcris(W_f\otimes W_f(j-\chi))$ we have introduced above:  We have
$$\langle\omega_{\rm tg}\rangle^{\perp,\circ}_\lambda={\rm span}_{\cO}\left\{\frac{p^{k+1}}{\lambda}\eta_f^\lambda\otimes \omega_{f\otimes\chi^{-1}}\right\}$$ 
and 
$$\mathcal{N}_{\lambda}^\circ={\rm span}_{\cO}\left\{\frac{p^{k+1}}{\lambda}\eta_f^\lambda\otimes \omega_{f\otimes\chi^{-1}}, \eta_f^\lambda\otimes \eta_{f\otimes\chi^{-1}}\right\}.$$
\end{remark}

\begin{example}\label{example}
Based on the discussion above, we see that
$$||\eta_f^{\lambda}\otimes\omega_f||_{\Fil^0\Dcris(R_f\otimes R_f(j-\chi))}=||\eta_f^{\lambda}\otimes\omega_f||_{\mathcal{N}_{\pm\alpha}^\circ}=||\eta_f^{\lambda}||_{\Dcris(R_f)}=p^{\frac{k+1}{2}}$$
whereas we have
$$||(\eta_f^{\alpha}+\eta_f^{-\alpha})\otimes\omega_f||_{\Fil^0\Dcris(R_f\otimes R_f(j-\chi))}=1\,.$$
\end{example}

\begin{proposition}
\label{prop_whenisBFnontrivialmodp}
We let
$$\langle\sim,\sim\rangle:\Dcris\left(R_f^*\otimes R_f^*(1-j+\chi)\right)\times \Dcris\left(R_f\otimes R_f(j-\chi)\right)\rightarrow\Dcris(\cO(1))\stackrel{\sim}{\lra} \cO $$
denote the natural crystalline pairing and
 $$\log:H^1_{\rm f}\left(\Qp,R_f^*\otimes R_f^*(1-j+\chi)\right)\longrightarrow{\Dcris(R_f^*\otimes R_f^*(1-j+\chi))}\big{/}{\Fil^0}$$
  denote the inverse of Bloch-Kato's exponential map as given in Lemma~\ref{lem:isoBK}.
 Then the following two conditions are equivalent:
\item[(I)] $\res_p(\overline{\BF}_1)\in H^1(\QQ_p,\overline{T})$ is non-zero.
\item[(II)] There exists $v\in \Fil^0\Dcris(W_f\otimes W_f(j-\chi))$ such that 
$$||\langle \log\left(\res_p( \BF_1)\right), v\rangle||_p=||v||_{\Fil^0\Dcris(R_f\otimes R_f(j-\chi))}\,.$$
\end{proposition}
\begin{proof}
Observe that we have
\begin{align*}
\res_p(\overline{\BF}_1)=0 \stackrel{\rm Lemma\,\ref{lem:isoBK}}{\iff}& 
 \log\left(\res_p( \BF_1)\right)\in \frak{m}\left({\Dcris(R_f^*\otimes R_f^*(1-j+\chi))}\big{/}{\Fil^0}\right)\\
 \iff & \langle \log\left(\res_p( \BF_1)\right), v^\circ\rangle \in \frak{m}, \forall v^\circ\in \Fil^0\Dcris(R_f\otimes R_f(j-\chi))\\
  \iff &||\langle \log\left(\res_p( \BF_1)\right), v\rangle||_p < ||v||_{W^\circ}, \forall v\in W\, ,
\end{align*}
where $W$  is a shorthand for ${\Fil^0\Dcris(W_f\otimes W_f(j-\chi))}$  and $W^\circ$ for the lattice ${\Fil^0\Dcris(R_f\otimes R_f(j-\chi))}$. The proof for the equivalence of  conditions (I) and (II) follows since we readily know that
$$||\langle \log\left(\res_p( \BF_1)\right), v\rangle||_p\leq  ||v||_{W^\circ}, \forall v\in W.$$ 
\end{proof}

\subsection{$p$-adic periods and Perrin-Riou's conjecture}
\label{subsec_ColmezVsKLZ}
Our goal in this subsection is to give a criterion of the non-triviality of $\overline{\BF}_1$ in terms of the $p$-adic periods of Perrin-Riou.  Fix throughout an even integer $j\in [1,k+1]$.
\begin{defn}
\item[(i)] We let $\Omega_\infty(j)$ denote the archimedean period  associated to the motive $M(f\otimes f\otimes\chi^{-1})(j)$, given as in \cite[Definition 7.1.12]{KLZ1}.
\item[(ii)] For $\lambda\in\{\pm\alpha\}$, we let $\Omega_p(j,\underline{v},\cN_\lambda)$ denote Perrin-Riou's $p$-adic period associated to the motive $M(f\otimes f\otimes\chi^{-1})(j)$ and a basis $\underline{v}=\{v_1,v_2\}$ of $\cN_\lambda$, given as in \cite[\S2.8]{ColmezBourbaki1998} and \cite[\S7.2]{KLZ1}. 
\item[(iii)] For any basis $\underline{v}^+$ of the optimized lattice $\cN_\lambda^\circ$, we set $\Omega_p(j, \cN_\lambda^\circ):=\Omega_p(j,\underline{v}^+,\cN_\lambda)$ and call it the optimized $p$-adic period. Note that the optimized period is well-defined only up to multiplication by an element in $\cO^\times$.
\item[(iv)]
For $\lambda\in\{\pm\alpha\}$, we write $L_p(f^\lambda,f^\lambda\otimes\chi^{-1},s)$ for the $(f^\lambda,f^\lambda\otimes\chi^{-1})$-specialization of the geometric $p$-adic $L$-function of Loeffler-Zerbes defined in \cite[\S9.3]{LZ1}. Here, $f^\lambda$ denotes the $p$-stabilization of $f$ with respect to $\lambda$. 
\end{defn}

\begin{remark}
\item[(i)] The definition of the Beilinson's archimedean period depends on, among other things, a choice of a basis for the ``cotangent space'' ${\rm Fil}^{1-j} M_{\rm dR}(f\otimes f\otimes\chi^{-1})^*$ (which is one-dimensional in this case),  as well as a generator of the space denoted by 
$$H_{\cH}^1({\rm Spec}\,\mathbb{R},M_B(f\otimes f\otimes\chi^{-1})^*(1-j))$$ in \cite{KLZ1}, which arises as the Hodge realization of a motivic class\footnote{In \cite{KLZ1}, this generator class is taken as the image of the Rankin-Eisenstein class under the Abel-Jacobi map for absolute Hodge cohomology.}. We note that 
$$H_{\cH}^1({\rm Spec}\,\mathbb{R},M_B(f\otimes f\otimes\chi^{-1})^*(1-j))\cong \ker\left(\alpha^{\rm Betti}_{M(f\otimes f\otimes\chi^{-1})(j)}\right)^*$$
where
$$\alpha_{M(f\otimes f\otimes\chi^{-1})(j)}^{\rm Betti}: M_B(f\otimes f\otimes\chi^{-1})^+(j)\otimes\mathbb{R}\lra M_{\rm dR}\left(f\otimes f\otimes\chi)(j)\right)\big{/}{\rm Fil}^0\otimes \mathbb{R}$$
is induced from the comparison isomorphism between the Betti and de Rham cohomology. 
The archimedean period (which is well-defined only up to multiplication by an element of $L^\times$) is non-zero and its dependence on these choices is linear.
\item[(ii)] Perrin-Riou's $p$-adic period for $M(f\otimes f\otimes\chi^{-1})(j)$ depends not only on the choice of $\underline{v}$, but also a basis of  ${\rm Fil}^{1-j} M_{\rm dR}(f\otimes f\otimes\chi^{-1})^*$, as well as a generator of the one-dimensional vector space denoted by $\ker\left(\alpha^{\rm cris}_{M(f\otimes f\otimes\chi^{-1},\cN_\lambda)(j)}\right)^*$ in \cite{KLZ1}. Here, $ \alpha^{\rm cris}_{M(f\otimes f\otimes\chi^{-1})(j),\cN_\lambda}$ is the compositum of the arrows 
\begin{align*}
 \alpha^{\rm cris}_{M(f\otimes f\otimes\chi^{-1})(j),\cN_\lambda}:\, \cN_\lambda\hookrightarrow &D_{\rm cris}(W_f\otimes W_f\otimes\chi^{-1})(j)\\
 &\lra M_{\rm dR}\left(f\otimes f\otimes\chi)(j)\right)\big{/}{\rm Fil}^0\otimes_L E\,.
\end{align*}
Namely, it is the map induced from the Faltings-Tsuji comparison isomorphism, restricted to the subspace $\cN_\lambda$.  

Suppose that $L_p(f^\lambda,f^\lambda\otimes\chi^{-1},j)\neq 0$. It follows from \cite[Theorem 6.5.9]{KLZ1} that the the one-dimensional $E$-vector space $\ker\left(\alpha^{\rm cris}_{M(f\otimes f\otimes\chi^{-1},\cN_\lambda)(j)}\right)^*$  is generated by the image of the syntomic Abel-Jacobi map.  If one takes the same motivic class that was used in {\rm (i)}, it turns out that the quantity 
$$\iota_p\left(\frac{L^\prime(f\otimes f\otimes\chi^{-1},j)}{\Omega_\infty(j)}\right)\times \Omega_p(j,\underline{v},\cN_\lambda)$$ depends only on $\cN_\lambda$ and $\underline{v}$. 
\item[(iii)] Still assuming  $L_p(f^\lambda,f^\lambda\otimes\chi^{-1},j)\neq 0$ and that the choices of bases for the vector spaces that intervene both in the definition of the archimedean and the definition of $p$-adic period are the same, we see that the quantity 
$$v_p\left(\iota_p\left(\frac{L^\prime(f\otimes f\otimes\chi^{-1},j)}{\Omega_\infty(j)}\right)\times \Omega_p(j,\cN_\lambda^\circ)\right)$$
depends only on $\cN_\lambda^\circ$.
\item[(iv)] As explained in \cite[\S2.8]{ColmezBourbaki1998}, the dependence of $\Omega_p(j,\underline{v},\cN_\lambda)$ on $\underline{v}$ is only through $v_1\wedge v_2$.
\end{remark}

\begin{theorem}
\label{THM_nonzeromodperiodimpliesBFmodpnonzero}
Suppose $j\in [1,k+1]$ is an even integer. Assume that the quantities $\iota_p\left(L^\prime(f\otimes f\otimes\chi^{-1},j)\big{/}\Omega_\infty(j)\right)$ and the optimized $p$-adic period $\Omega_p(j,\cN^\circ_\lambda)$ are $p$-adic units for at least one of the choices $\lambda\in \{\pm\alpha\}$. Then \eqref{eqn:vanishingofstrictselmer} holds true.
\end{theorem}
Note that we may ensure that $\iota_p\left(L^\prime(f\otimes f\otimes\chi^{-1},j)\big{/}\Omega_\infty(j)\right)$ is a $p$-adic unit by choosing $p$ sufficiently large.
\begin{proof} We drop $\iota_p$ throughout this proof for the sake of notational simplicity. We also note that much of our notation here is borrowed from \cite{KLZ1}. For $x,y\in \mathbb{C}_p^\times$, let us write $x\approx y$ if $xy^{-1}\in \cO_{\mathbb{C}_p}^\times$. It follows from \cite[Theorem~6.5.9 and Remark~6.5.10(iii)]{KLZ1} (see also \cite[Theorem 3.3.4]{KLZ2}) that
\begin{align*}
\langle\log\left(\res_p( \BF_1)\right), \eta_f^{\pm\lambda}\otimes \omega_f\rangle&\approx\frac{{k \choose j}\,k!\,\mathcal{E}(f^{\lambda}){\mathcal{E}^*(f^{\lambda})}
}{\mathcal{E}(f^{\lambda},f^{\lambda}\otimes \chi^{-1},j)}L_p(f^{\lambda},f^{\lambda}\otimes\chi^{-1},j)\\
&\approx \frac{\mathcal{E}(f^{\lambda}){\mathcal{E}^*(f^{\lambda})}
L_p(f^{\lambda},f^{\lambda}\otimes\chi^{-1},j)}{\mathcal{E}(f^{\lambda},f^{\lambda}\otimes \chi^{-1},j)}\,
\end{align*}
since $p>2k+2$ and $p\nmid N_f$. This combined with \cite[Theorem 7.2.6]{KLZ1} yields (again relying on the fact that $p>2k+2$)
\begin{align*}
\langle \log\left(\res_p( \BF_1)\right),\eta_f^{\lambda}\otimes \omega_f\rangle&\approx \frac{L^\prime_{\{p\}}(f\otimes f\otimes\chi^{-1},j)}{\mathcal{E}(f^{\lambda},f^{\lambda}\otimes \chi^{-1},j)}\times \frac{\widetilde{\Omega}_p(j,\underline{v},\cN_{\lambda})}{\Omega_\infty(j)}\,.
\end{align*}
where $\widetilde{\Omega}_p(j,\underline{v},\cN_{\lambda})=\Omega_p(j, \frac{1-p^{-1}\varphi^{-1}}{1-\varphi}(-t)^{-j}\underline{v},\cN_{\lambda})$. Here, we have taken $v_1=\eta_f^{\lambda}\otimes\omega_{f\otimes\chi^{-1}}$ and $v_2=\eta_f^{\lambda}\otimes\eta_{f\otimes\chi^{-1}}$ so the $p$-adic periods here agree (up to $p$-adic units) with those appear in op. cit. multiplied by $\mathcal{E}(f^{\lambda}){\mathcal{E}^*(f^{\lambda})}$. 

The computation in \cite[\S7.2]{KLZ2} (see in particular (7.2.3) and the remark in Definition~ 7.2.5) shows that 
$$\frac{L^\prime_{\{p\}}(f\otimes f\otimes\chi^{-1},j)}{\mathcal{E}(f^{\lambda},f^{\lambda}\otimes \chi^{-1},j)}\times \frac{\widetilde{\Omega}_p(j,\underline{v},\cN_{\lambda})}{\Omega_\infty(j)}=\frac{L^\prime(f\otimes f\otimes\chi^{-1},j)}{\Omega_\infty(j)}\times \Omega_p(j,\underline{v},\cN_{\lambda})\,.$$

This implies 
\begin{align*}
\langle\log\left(\res_p( \BF_1)\right), \eta_f^{\lambda}\otimes \omega_f\rangle&\approx\frac{L^\prime(f\otimes f\otimes\chi^{-1},j)}{\Omega_\infty(j)}\times \Omega_p(j,\underline{v},\cN_{\lambda})\\
&\approx \frac{\lambda}{p^{k+1}}\times\frac{L^\prime(f\otimes f\otimes\chi^{-1},j)}{\Omega_\infty(j)}\times \Omega_p(j,\cN_{\lambda}^\circ)\,,
\end{align*}
where we recall that $\underline{v}=\{\eta_f^{\lambda}\otimes\omega_{f\otimes\chi^{-1}},\eta_f^{\lambda}\otimes\eta_{f\otimes\chi^{-1}}\}$ and the second equality follows from the description of a basis for the lattice $\cN^\circ_\lambda$  in Remark~\ref{rem_basisfortheoptimisedlattice}.

Under our running assumptions that $L^\prime(f\otimes f\otimes\chi^{-1},j)\big{/}\Omega_\infty(j)$ and $\Omega_p(j,\cN^\circ_\lambda)$ be $p$-adic units for some choice of $\lambda\in \{\pm\alpha\}$, we conclude that
$$||\langle\log\left(\res_p( \BF_1)\right), \eta_f^{\lambda}\otimes \omega_f\rangle ||_p=||\eta_f^{\lambda}\otimes \omega_f||_{{\Fil^0\Dcris(R_f\otimes R_f(j-\chi))}}\,.$$

The result follows from Proposition~\ref{prop_BFnontrivialmodpimpliestrivialselmer}.
\end{proof}

\subsection{A numerical criterion}

It is possible to check the non-vanishing of $\res_p(\overline{\BF}_1)\in H^1(\QQ_p,\overline{T})$ by evaluating the $p$-adic $L$-values  $L_p(f^{\lambda},f^{\lambda}\otimes\chi^{-1},j)$ for $\lambda\in\{\alpha,-\alpha\}$:

\begin{lemma}
\label{lem:NumericalCriterion}
Suppose that $v_p(L_p(f^{\lambda},f^\lambda\otimes\chi^{-1},j))={-5k/2+2j-7/2}$ for either $\lambda=\alpha$ or $-\alpha$.
Then $\res_p(\overline{\BF}_1)\in H^1(\QQ_p,\overline{T})$ is non-zero.
\end{lemma}

\begin{proof}In the notation of the proof of Theorem~\ref{THM_nonzeromodperiodimpliesBFmodpnonzero}, we calculate explicitly that
\[\cE(f^{\lambda},f^{\lambda}\otimes\chi^{-1},j)\approx {p^{-2(k-j+2)}},\]
 $\cE(f^{\lambda})=1+\frac{1}{p}$ and $\cE^*(f^{\lambda})=2$. 
Recall from \cite[Theorem~6.5.9 and Remark~6.5.10(iii)]{KLZ1} that 
\[
\langle \log\left(\res_p( \BF_1)\right), \eta_f^{{\lambda}}\otimes \omega_f\rangle\approx\frac{ \cE(f^{\lambda})\cE(f^{\lambda})L_p(f^{\lambda},f^{\lambda}\otimes\chi^{-1},j)}{\cE(f^{\lambda},f^{\lambda}\otimes\chi^{-1},j)}.
\]
 The result now follows from Proposition~\ref{prop_whenisBFnontrivialmodp} and Example~\ref{example}.
\end{proof}

\begin{remark}\label{remarkatlimbo}
\item[(i)] It follows on comparing with the proof of Theorem~\ref{THM_nonzeromodperiodimpliesBFmodpnonzero}, the rather random looking numerical condition in  Lemma~\ref{lem:NumericalCriterion} hold true (for $p>>0$) if and only if $\Omega_p(j,\mathcal{N}_{\lambda}^\circ)$ is a $p$-adic unit. 
\item[(ii)]The proof of Lemma~\ref{lem:NumericalCriterion} also shows that if 
$$v_p\left(L_p(f^{\alpha},f^\alpha\otimes\chi^{-1},j)+L_p(f^{-\alpha},f^{-\alpha}\otimes\chi^{-1},j)\right)={-2k+2j-3},$$
then $\res_p(\overline{\BF}_1)\in H^1(\QQ_p,\overline{T})$ is non-zero.
\item[(iii)] {Due to the absence of cricial values, it is not possible to relate the aforementioned $p$-adic $L$-values to complex $L$-values.} In order to evaluate them numerically, one may attempt to extend the algorithms of Lauder in \cite{lauder} to our setting. According to our limited understanding on the matter, the main computational difficulty is working in an efficient manner with spaces of nearly overconvergent (but 
not necessarily overconvergent) modular forms.
\end{remark}

\section*{Acknowledgment}
We thank Alan Lauder and David Loeffler for answering our questions during the preparation of this paper. We also thank the anonymous referee for a very thorough reading of an earlier version of this article and for many useful suggestions. K.B. has received funding from the European Union's Horizon 2020 research and innovation programme under the Marie Sk\l odowska-Curie Grant Agreement No. 745691 (CriticalGZ). A.L. is supported by the NSERC Discovery Grants Program 05710.

\bibliographystyle{amsalpha}
\bibliography{references}
\end{document}